\documentclass[letterpaper, 10 pt, conference]{ieeeconf}  

\IEEEoverridecommandlockouts                              

\usepackage{amsfonts}
\usepackage{amsmath}
\usepackage{amssymb}
\usepackage{color}
\usepackage{graphicx}
\usepackage{float}
\usepackage[mathscr]{eucal}
\usepackage{amsfonts}
\usepackage{float}
\usepackage[colorlinks,linkcolor=blue]{hyperref}
\usepackage[latin1]{inputenc}
\usepackage{graphicx}
\usepackage{amssymb}
\usepackage{verbatim}
\usepackage{epsfig}
\usepackage[labelfont=bf]{caption}
\usepackage{enumerate}
\usepackage{subfig}
\usepackage{epstopdf}

\newtheorem{lemma}{Lemma}
\newtheorem{proposition}{Proposition}
\newtheorem{corollary}{Corollary}
\newtheorem{fact}{Fact}

\newtheorem{remark}{Remark}

\newtheorem{assumption}{Assumption}

\def\begcen{\begin{center}}
	\def\endcen{\end{center}}

\def\caly{{\cal Y}}

\def\hatthe{\hat{\eta}}
\def\tilthe{\tilde{\eta}}

\def\L2{{\cal L}_2}
\def\L2e{{\cal L}_{2e}}

\def\rea{\mathbb{R}}

\def\sign{\mbox{sign}}

\def\hatthe{\hat{\theta}}
\def\tilthe{\tilde{\theta}}

\def\begali#1{\begin{align}{#1}\end{align}}

\def\begequarr{\begin{eqnarray}}
	\def\endequarr{\end{eqnarray}}
\def\begequarrs{\begin{eqnarray*}}
	\def\endequarrs{\end{eqnarray*}}
\def\begarr{\begin{array}}
	\def\endarr{\end{array}}
\def\begequ{\begin{equation}}
	\def\endequ{\end{equation}}
\def\lab{\label}
\def\begdes{\begin{description}}
	\def\enddes{\end{description}}
\def\begenu{\begin{enumerate}}
	\def\begite{\begin{itemize}}
		\def\endite{\end{itemize}}
	\def\endenu{\end{enumerate}}

\def\lef[{\left[\begin{array}}
	\def\rig]{\end{array}\right]}

\def\begcen{\begin{center}}
	\def\endcen{\end{center}}
\def\begrem{\begin{remark}\rm}
	\def\endrem{\end{remark}}
\def\begassum{\begin{assumption}}
	\def\endassum{\end{assumption}}
\def\begassums{\begin{assumption*}}
	\def\endassums{\end{assumption*}}
\def\begassu{\begin{ass}}
	\def\endassu{\end{ass}}
\def\beglem{\begin{lemma}}
	\def\endlem{\end{lemma}}
\def\begcor{\begin{corollary}}
	\def\endcor{\end{corollary}}
\def\begfac{\begin{fact}}
	\def\endfac{\end{fact}}
\def\TAC{{\it IEEE Trans. Automat. Contr.}}
\def\AUT{{\it Automatica}}

\def\hatthe{\hat{\theta}}
\def\tilthe{\tilde{\theta}}

\def\L2e{{\cal L}_{2e}}

\def\rea{\mathbb{R}}
\def\intnum{\mathbb{Z}}
\def\sign{\mbox{sign}}

\def\IJACSP{{ Int. J. on Adaptive Control and Signal Processing}}

\def\TAC{{ IEEE Trans. Automatic Control}}

\def\EJC{{ European Journal of Control}}

\def\AUT{{ Automatica}}

\title{\LARGE \bf
Parameter Identification with Finite-Convergence Time Alertness Preservation*
}

\author{Romeo Ortega$^{1}$, Alexey Bobtsov$^{2}$ and Nikolay Nikolaev$^2$
\thanks{*This paper is partially supported by the Ministry of Science and Higher Education of Russian Federation, passport of goszadanie no. 2019-0898.}
\thanks{$^{1}$Romeo Ortega is with Departamento Acad\'{e}mico de Sistemas Digitales, ITAM, Ciudad de M\'exico, M\'{e}xico.
        {\tt\small romeo.ortega@itam.mx}}%
\thanks{$^{2}$Alexey Bobtsov and Nikolay Nikolaev are with ITMO University, Kronverkskiy av. 49, St. Petersburg, 197101, Russia.
        {\tt\small  bobsov@mail.ru and nikona@yandex.ru}}%
}

\begin{document}

\maketitle
\thispagestyle{empty}
\pagestyle{empty}

\begin{abstract}

In this brief note we present two new parameter identifiers whose estimates converge in finite time under weak interval excitation assumptions. The main novelty is that, in contrast with other finite-convergence time (FCT) estimators, our schemes preserve the FCT property when the parameters change. The previous versions of our FCT estimators can track the parameter variations only asymptotically. Continuous-time and discrete-time versions of the new estimators are presented.

\end{abstract}


\section{PARAMETER ESTIMATORS WITH ALERTNESS-PRESERVING FINITE-CONVERGENCE TIME}
\lab{sec1}
%
The objective of this work is to propose continuous-time (CT) and discrete-time (DT) {\em on-line} estimators of the parameters $\theta \in \rea^q$ of a linear regression equation (LRE) of the form 
\begequ
\lab{veclre}
y=\phi^\top \theta,
\endequ 
from the measurable quantities $y$ and $\phi$. The estimators should satisfy three specifications
\begenu
\item[{\bf S1}] Convergence of the estimates should be achieved in {\em finite-time}.
\item[{\bf S2}] Convergence is ensured under weak {\em interval excitation (IE)} conditions \cite{KRERIE}.
\item[{\bf S3}] The estimator should preserve its {\em alertness} to be able to estimate---still with FCT---future variations of the unknown parameters.
\endenu

Instrumental to solve this problem is the use of the dynamic regressor extension and mixing (DREM) procedure proposed in \cite{ARAetal_tac17}. In particular, we adapt the FCT-DREM estimator proposed in \cite{ORTetal_aut20} to incorporate the new feature of FCT alertness preservation (AP). A CT version of such an scheme was reported in \cite[Section V]{ORTetal_tac20}, in this note we give a DT version of it. 

An early variation of the least-squares method, that converges in finite time, was proposed 32 years ago. Similarly to more recent FCT estimators \cite{CHOetal_tac18,CHOetal_ijacsp13,ROYetal_tac18}  in its initial stage the algorithm is akin to an off-line estimator, which involves a numerically sensitive matrix inversion and, as it converges to a standard least-squares, loses its alertness. For the sake of completeness we present comparative simulations of the proposed estimator with the FCT estimator reported in \cite{WANetal_tac20}, which relies on the injection of high-gain via the use of fractional powers and/or relays in the estimator dynamics. See also \cite{RIOetal_tac17,WANetal_ejc20} where a similar high-gain approach is adopted.\\

\noindent {\bf Notation.} $\rea_{>0}$, $\rea_{\geq 0}$, $\intnum_{>0}$ and $\intnum_{\geq 0}$ denote the positive and non-negative real and integer numbers, respectively. Continuous-time (CT) signals $s:\rea_{\geq 0} \to \rea$ are denoted $s(t)$, while for discrete-time (DT) sequences $s:\intnum_{\geq 0} \to \rea$ we use $s(k):=s(kT_s)$, with $T_s \in \rea_{> 0}$ the sampling time. When a formula is applicable to CT signals and DT sequences the time argument is omitted. 
%

\section{FCT ESTIMATORS: FORMULATION FROM SCALAR LRES VIA DREM}
\lab{sec2}
As it has been widely documented the powerful DREM estimator design procedure \cite{ARAetal_tac17} allows us to generate, from the $q$-dimensional LRE \eqref{veclre}, $q$-scalar LREs of the form
\begequ
\label{scalre}
\caly_i = \Delta \theta_i,\quad i \in \bar q:=\{1,2,\dots,q\}
\endequ
where $\Delta$ is the determinant of an extended regressor matrix. In the remaining of the note, we will use this simple scalar LREs to design the AP-FCT parameter estimator. 
\subsection{CT FCT-DREM}
\lab{subsec21}
The following CT FCT-DREM estimator was reported in  \cite{ORTetal_aut20} and, as it constitutes the basis of our new AP-FCT, we repeat it for ease of reference. Also, to make the note self-contained we give a brief summary of the proof, referring the interested reader to  \cite{ORTetal_aut20} for further details.

\begin{proposition}\em
	\lab{pro1}
	Consider the scalar CT LREs  \eqref{scalre} and the gradient-descent estimators\footnote{In the sequel, the quantifier $i \in \bar q$ is omitted for brevity.} 
	\begequ
	\lab{parest}
	\dot{\hat{\theta}}_i(t) = \gamma_i\Delta(t)[\caly_i(t) - \Delta(t) \hat\theta_i(t)],
	\endequ
	with $\gamma_i >0$. Define the FCT estimate
	\begequ
	\lab{hatthew}
	\hat \theta_i^{\tt FCT}(t) :={1  \over 1 -w^{\tt c}_i(t)}[ \hat \theta_i(t) - w_i^{\tt c}(t)\hat \theta_i(0)],
	\endequ
	where $w_i^{\tt c}(t)$ is defined via the clipping functions
	\begequ
	\lab{condit}
	w_i^{\tt c}(t)=\left \{ \begin{array}{lcr} \mu_i  \;\;\; \mbox{if} \;\;\;\;\;w_i(t) \geq \mu_i\\ \\ w_i(t)\;\; \mbox{if} \;\;\;\;\; w_i(t) <\mu_i,
	\end{array} \right.
	\endequ
	$\mu_i \in (0,1)$ are designer chosen parameters, and $w_i(t)$ is given by 
	\begequ
	\lab{wi}
	\dot{w}_i(t)=-\gamma_i\Delta^2(t) w_i(t),\;w_i(0) =1.
	\endequ
	
	Assume there exists a time $t_i^c>0$ such that the IE condition \cite{KRERIE}
	\begequ
	\lab{sufexc}
	\int^{t_i^c}_0 \Delta^2(\tau)d\tau\geq-\frac{1}{\gamma_i}\ln(1-\mu_i),
	\endequ
	is satisfied. Then,
	\begequ
	\lab{hatequx1}
	\hat \theta_i^{\tt FCT}(t)=\theta_i,~~\forall t>t_i^c,
	\endequ
	that is, the estimator has the feature of FCT.
	
\end{proposition}

\begin{proof}
	Define the parameter errors $\tilde \theta_i:=\hat \theta_i-\theta_i.$ The scalar error equations are given by
	$$
	\dot {\tilde \theta}_i(t)=-\gamma_i \Delta^2(t) \tilde\theta_i(t),
	$$
	whose explicit solution is 
	\begali{
		\lab{solctpee}
		\tilde \theta_i(t) = e^{-\gamma_i \int_0^t \Delta^2(s)ds} \tilde \theta_i(0).
	}
	Now, notice that the solution of \eqref{wi} is
	$$
	w_i(t)=e^{-\gamma_i \int_0^{t} \Delta^2(s)ds}.
	$$
	The key observation is that, using the equation above in \eqref{solctpee}, and rearranging terms we get that 
	\begequ
	\lab{keyrel}
	[1-w_i(t)]\theta_i= \hat \theta_i(t) - w_i(t) \hat \theta_i(0).
	\endequ
	Now, observe that $w_i(t)$ is a non-increasing function and, under the interval excitation assumption \eqref{sufexc}, we have that 
	\begequ
	\lab{wlc}
	w^{\tt c}_i(t)=w_i(t) < \mu_i<1,\;\forall t \geq t_c,
	\endequ
	Clearly, \eqref{keyrel} and \eqref{wlc} imply that 
	$$
	\frac{1}{1-w_c(t)}\left[\hat \theta(t)-w_c(t)\hat\theta(0)\right]=\theta,\;\forall t > t_c,
	$$
	completing the proof.
\end{proof} 

\subsection{DT FCT-DREM}
\lab{subsec22}
In  \cite[Proposition 2]{ORTetal_aut20} a DT DREM estimator was reported. We give below an FCT-version of this scheme  and give a brief summary of the proof. 

\begin{proposition} \em
	\label{pro2}
	Consider the scalar DT LREs  defined by \eqref{scalre} with the DT gradient-descent estimator   
	\begin{equation}
		\label{hatthe}
		\hat \theta_i(k+1)=\hat \theta_i(k)+\frac{{\Delta}(k)}{c_i+{\Delta}^2(k)} \left[ {\caly}_i (k+1) - {\Delta}(k)  \hat \theta_i(k)\right]
	\end{equation}
	with positive constants $c_i$. Define the dynamic extension
	\begin{equation}
		\label{wkplu1}
		w_i(k+1)= \frac{c_i}{c_i+{\Delta}^2(k)} w_i(k), \;\;w_i(0)=1.
	\end{equation}
	and the clipping function 
	\begequ
	\lab{clifun}
	w_i^c(k) = \left\{
	\begin{aligned}
		& \rho_i & \mbox{if $w_i(k)\in [\rho_i,1]$}\\
		&\\
		&w_i(k) &  \mbox{if $w_i(k)\in[0,\rho_i)$},
	\end{aligned}
	\right.
	\endequ
	where $\rho_i \in (0,1)$ are designer chosen constants. Assume there exists a $k_i^c \in (0,\infty)$ such that the IE condition
	\begequ
	\lab{intexcass}
	\prod_{i=0}^{k_i^c}\left[  \frac{c_i}{c_i +{\Delta}^2(i)} \right] < \rho_i,
	\endequ
	is satisfied. Then, 
	\begin{equation}
		\lab{thefct}
		\hat{\theta}_i^{\tt FCT}(k):= {1 \over 1 - w_i^c(k)}[ \hat \theta_i(k) - w_i^c(k)\hat \theta_i(0)],
	\end{equation}
	ensures
	$$
	{\hat \theta}_i^{\tt FCT}(k)=\theta_i,\;\forall k \geq k_i^c.
	$$
\end{proposition}

\vspace{0.3cm}

\begin{proof}
	From,   \eqref{scalre} and \eqref{hatthe} we get the parameter error equation 
	\begin{equation}
		\label{dtilthe}
		\tilde \theta_i(k+1)= \frac{c_i}{c_i +{\Delta}^2(k)} \tilde \theta_i(k)
	\end{equation}  
	whose explicit  solution satisfies 
	\begin{equation}
		\label{tilthe}
		\tilde \theta_i(k)= \psi_i(k) \tilde \theta_i(0),
	\end{equation}
	where, for ease of future reference, we defined the scalar sequence
	\begequ
	\lab{psik}
	\psi_i(k):=\prod_{i=0}^{k}\left[  \frac{c_i}{c_i +{\Delta}^2(i)} \right]
	\endequ
	The solution of \eqref{wkplu1}  is given by
	\begequ
	\lab{solwk}
	w_i(k)= \psi_i(k),
	\endequ
	whose replacement in \eqref{tilthe} yields
	$$
	\tilde \theta_i(k)=w_i(k)\tilde \theta_i(0). 
	$$
	Using the definition of the parameter error and rearranging terms we get that
	\begequ
	\lab{fctdt}
	[1 - w_i(k)]\theta_i = \hat \theta_i(k) - w_i(k)\hat \theta_i(0).
	\endequ
	According to \eqref{clifun} we have that, under the assumption \eqref{intexcass}, $w_i(k)<\rho_i<1$ for all $k \geq k_i^c$, Consequently, for $k \geq k_i^c$ we can write
	$$
	\theta_i ={1 \over 1 - w_i(k)}[ \hat \theta_i(k) - w_i(k)\hat \theta_i(0)].
	$$
	The proof is completed, from \eqref{thefct}, noting that $w_i(k)=w_i^c(k)$ for all $k \geq k_i^c$.
\end{proof}
\section{NEW ESTIMATOR WITH FCT ALERTNESS PRESERVATION}
\lab{sec3}
%
There are two practical problems with the approach described above. First, the estimates at the current time are reconstructed from the knowledge of the {\em initial} estimate $\hatthe_i(0)$, complicating the task of tracking variations of the true parameters after convergence of the estimates. Second,  {\em independently} of the behaviour of $\Delta$, the functions $w_i$ are monotonically non-increasing and converge to zero if $\Delta$ is not square summable (or integrable). In this case, $\hat \theta_i^{\tt FCT} \to \hat \theta_i$, and the FCT estimator converges to a standard gradient, losing its FCT feature.  Therefore, to keep the {\em FCT alertness} of the estimator, {\em i.e.}, to track parameter variations in finite-time upon the arrival of new excitation, it is necessary to reset the estimators---a modification that is always problematic to implement. A typical procedure is the so-called {\em covariance resetting} for least-squares algorithms, see \cite[Section 2.4.2]{SASBODbook}.

These drawbacks can be overcome with the new FCT-DREM estimators proposed below.
\subsection{CT FCT-DREM with alertness preservation}
\lab{subsec31}
For the sake of brevity, we present only the derivation of a relation similar to \eqref{keyrel}, from which we can easily construct the FTC estimator and prove that the new FTC estimator does not converge to the gradient one.

\begin{proposition}\em
	\lab{pro3}
	Fix $T_{\tt D} \in \rea_{> 0}$ and define 
	\begequ
	\lab{dotwd}
	\dot w_i^{\tt D}(t)=-\gamma_i\left[\Delta^2(t)-\Delta^2(t-T_{\tt D})\right]w_i^{\tt D}(t),\;w^{\tt D}_i(0)=1.
	\endequ
	Then,
	\begequ
	\lab{ttt}
	\left[1-w_i^{\tt D}(t)\right]\theta_i = \hatthe_i(t)-w_i^{\tt D}(t)\hatthe_i(t-T_{\tt D}).
	\endequ
	Moreover, $w_i^{\tt D}(t)$ is {\em bounded away from zero}.
\end{proposition}
\begin{proof}
	Without loss of generality we assume that $\Delta(t-T_{\tt D})=0$ for $t < T_{\tt D}$. Then, the solution of \eqref{dotwd} is
	\begequ
	\lab{solwd}
	w_i^{\tt D}(t)=e^{-\gamma_i\int_{t-T_{\tt D}}^t \Delta^2(s)ds}.
	\endequ
	From \eqref{solwd}, and the fact that 
	$$
	\int_{t-T_{\tt D}}^t \Delta^2(s)ds \leq \Delta^2_{\max}T_{\tt D},
	$$
	where $\Delta_{\max} \geq \|\Delta(t)||_\infty$, we conclude that 
	$$
	w_i^{\tt D}(t) \geq e^{-\gamma_i \Delta^2_{\max} T_{\tt D}} >0.
	$$
	Now, from the solution of the parameter error equation \eqref{solctpee} in the interval $[t-T_{\tt D},t]$ we get
	$$
	\tilde \theta_i(t) = e^{-\gamma_i \int_{t - t_{\tt D}}^t \Delta^2(s)ds} \tilde \theta_i(t - T_{\tt D}).
	$$
	Hence, $\tilthe_i(t)=w_i^{\tt D}(t)\tilthe_i(t-T_{\tt D}).$ The proof of the claim is established rearranging the terms of the equation above.
	
\end{proof}

\begrem
\lab{rem1}
It is important to note that when $\Delta(t)$ decreases---that is, when we loose excitation---$w_i^{\tt D}(t)$ grows towards one, and the alertness in not lost. On the other hand, when new excitation arrives and $\Delta(t)$ grows, then $w_i^{\tt D}(t)$ decays and the FCT condition $w_i^{\tt D}(t)<\mu_i$ is satisfied. In this way, the new FTC estimator preserves its FTC property if the parameters change. This fact is illustrated in the simulations of Section \ref{sec4}. Notice also that, in contrast with the FCT estimator of Proposition \ref{pro1}, where the calculation of the $\hat \theta^{\tt FCT}_i$ in \eqref{hatthew} is done using the {\em initial} parameter estimate $\hat \theta_i(0)$, the FCT reconstruction of the estimated parameter is done in \eqref{ttt} using the estimate at time $t-T_{\tt D}$, that is, $\hatthe_i(t-T_{\tt D})$.
\endrem

\begrem
\lab{rem2}
For the new FTC DREM estimator the interval excitation inequality becomes the existence of a time $t_i^c \geq T_{\tt D}$ such that
\begequ
\lab{inttcd}
\int^{t_i^c}_{t_i^c-T_{\tt D}} \Delta^2(s) ds \geq -\frac{1}{\gamma_i}\ln(1-\mu_i).
\endequ
Recalling \eqref{solwd}, it has the same interpretation as \eqref{sufexc}.
\endrem  
\subsection{DT FCT-DREM with alertness preservation}
\lab{subsec32}
Similarly to the CT case, for the sake of brevity, we present only the derivation of a relation similar to \eqref{fctdt}, from which we can easily construct the FTC estimator and prove that the new FTC estimator does not converge to the gradient one.

\begin{proposition}\em 
	\lab{pro4}
	Fix a positive integer ${\tt d}$. Consider the DT, scalar LRE \eqref{scalre} and the gradient parameter update \eqref{hatthe} with the dynamic extension \eqref{wkplu1} 
	and the sequence
	\begali{
		\lab{wdk}
		w_i^{\tt d}(k)&={w_i(k) \over w_i(k-{\tt d}-1)}.
	}
	Then,
	$$
	[1 - w_i^{\tt d}(k)]\theta_i = \hat \theta_i(k) - w_i^{\tt d}(k)\hat \theta_i(k-{\tt d}).
	$$
\end{proposition}

\vspace{0.3cm}

\begin{proof}
	First, we make the observation that, for all finite $k$, the sequence $w_i(k)$ is bounded away from zero. Replacing the solution of $w_i(k)$---given by \eqref{psik} and \eqref{solwk}---in \eqref{wdk} it is clear that
	$$
	w_i^{\tt d}(k)=\prod_{i=k-{\tt d}}^{k}\left[  \frac{c_i}{c_i +{\Delta}^2(i)} \right]
	$$ 
	Hence, replacing the equation above in \eqref{tilthe} we have that
	$$
	\tilthe_i(k)=w_i^{\tt d}(k)\tilthe_i(k-{\tt d}).
	$$
	The proof is completed rearranging the terms of the identity above.
\end{proof}

\begrem
\lab{rem3}
It is clear that the same behavior that is indicated in Remark \ref{rem1} regarding the CT version is observed in this DT one---explaining the important FCT-AP property. 
\endrem

%

\section{SIMULATIONS}
\lab{sec4}
%
In this section we present simulations illustrating the results of Propositions \ref{pro1}-\ref{pro4} for a single parameter. Moreover, for the sake of comparison, we show some simulation results of the high-gain FCT estimator proposed in \cite{WANetal_tac20}.
\subsection{CT FCT-DREM with alertness preservation}
\lab{subsec41}
In this subsection we compare the FCT DREM of  Proposition \ref{pro1} and the new FCT DREM of Proposition \ref{pro3}. The objective of the simulation is to prove that the new FTC DREM is able to react when new excitation arrives. This is in contrast with the old FCT DREM estimator that, since $w_i(t) \to 0$, converges to the gradient estimator and loses its FTC alertness property. 

We consider two scenarios: with and without {\em excitation} in $\Delta(t)$. For the first case we consider the PE signal $\Delta(t) = \sin({\pi \over 10} t)$, and for the second one $\Delta(t) = \frac{1}{\sqrt{t+1}}$. Note that in the second case $\Delta(t)\to 0$, hence it is not PE. However, $\Delta(t) \not \in \mathcal{L}_2$, hence it satisfies the conditions for convergence of the DREM estimator \cite{ARAetal_tac17,ORTetal_tac20}. 

To illustrate the FTC tracking capabilities of the estimators the unknown parameter $\theta$ is {\em time-varying} and given by
\[ 
\theta(t) = \begin{cases}
	10 \text{ for } 0\le t <10, \\
	15 \text{ for } 10\le t <20, \\		
	15-0.5(t-20) \text{ for } 20\le t <30, \\		
	10 \text{ for } t>30, \\
\end{cases}
\]
{\em i.e.}, it starts at $10$, jumps to $15$ at $t=10$, and then linearly returns to $10$.

For the simulations we set $\gamma = 2$, $\mu=0.98$, and $T_{\tt D}=0.2$. These parameters have been chosen such that the transients of both FCT estimators coincide in the ideal case when $\theta$ is constant and the system is excited. 

The transient of the estimators for  $\Delta(t) =  \sin({\pi \over 10} t)$ are given in Fig. \ref{fig:fig1_CT}, where we plot the time-varying parameter $\theta$, the gradient estimate $\hat \theta^{grad}(t)$, as well as the old and the new FCT estimates, denoted in the plots as $\hat \theta^{\tt FCT}(t)$  and $\hat \theta^{\tt FCT-D}(t)$, respectively.  We observe that, as expected, in the time interval $[0,10]$ both FCT estimators are overlapped and converge in finite time, while the gradient converges only asymptotically. The difference between the old FCT estimator and the new one is clearly appreciated in the time interval $[10,20]$, where we see that the new FCT estimator tracks the parameter variation in {\em finite time}, while the old one---now glued to the gradient---only does it asymptotically.  

\begin{figure}[tb]
	\centering
	\includegraphics[width=1\linewidth]{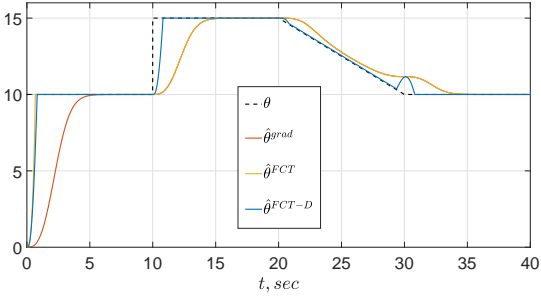}
	\caption{Transients of the CT parameter estimates with $\Delta(t)=\sin({\pi \over 10} t)$.}
	\label{fig:fig1_CT}
\end{figure}

\begin{figure}[tb]
	\centering
	\includegraphics[width=1\linewidth]{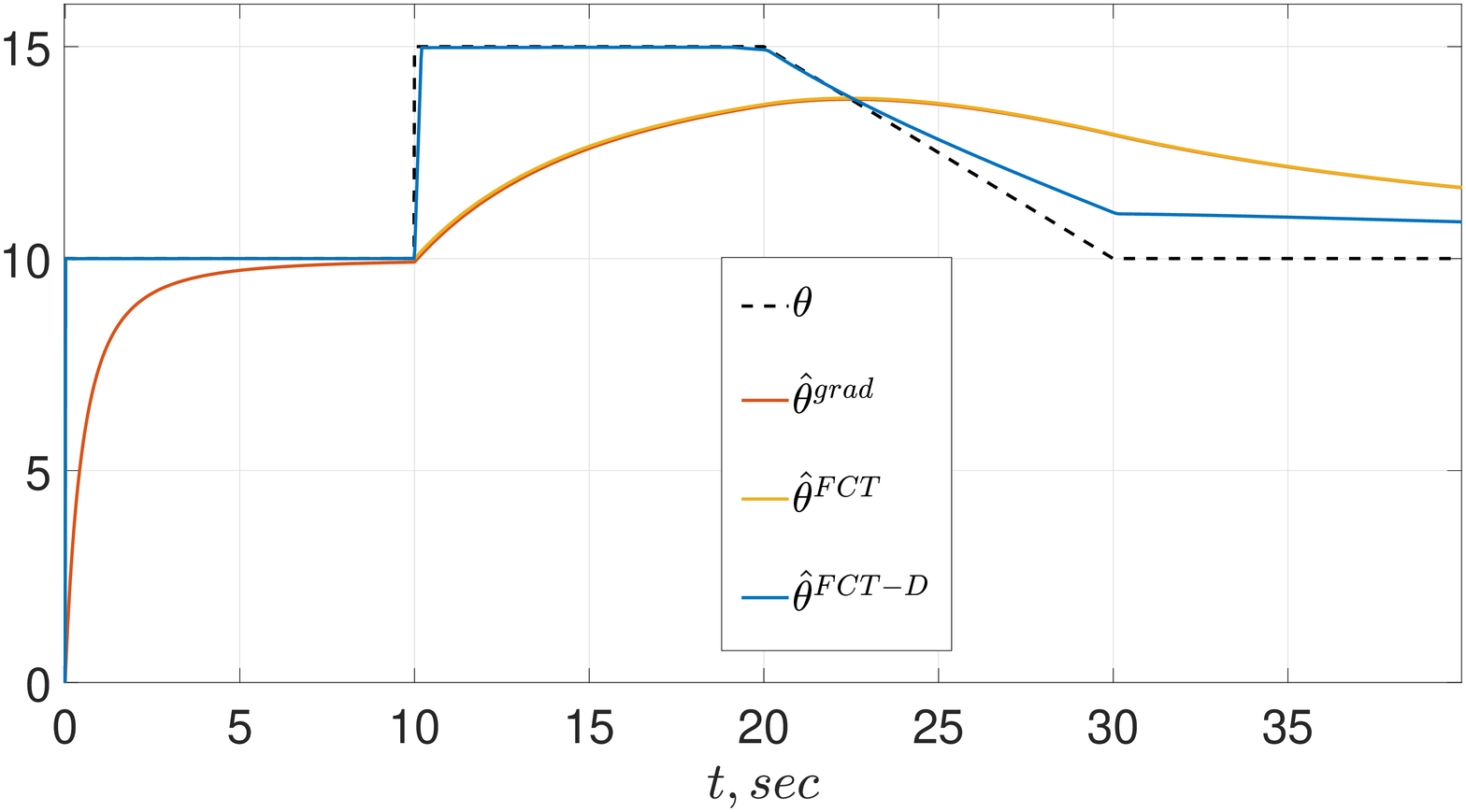}
	\caption{Transients of the CT parameter estimates with $\Delta(t)=\dfrac{1}{\sqrt{t+1}}$.}
	\label{fig:fig2_CT}
\end{figure}

The behavior of the CT estimators for $t \in [20,40]$ shows that, as predicted by the theory, the old FCT behaves as the gradient estimator and their trajectories coincide. On the other hand, the new estimator preserves FCT alertness after the first parameter jump and achieves fast tracking of the linearly time-varying $\theta(t)$. We also observe in the figure a blip in the estimates at $t=30$ that coincides with the time of freezing of the true parameter. 

For the non-PE case of $\Delta=1/\sqrt{t+1}$, the transients of the CT estimators are given in Fig. \ref{fig:fig2_CT}. We observe that both FCT estimators, again, essentially coincide in the first few seconds and converge in finite time, while the gradient does it only asymptotically. After the first parameter change at $t=10$ the old FTC and the gradient coincide, while the new FCT manages to track in finite time the parameter jump. However, during the ramp parameter change---because of the lack of excitation---neither one of the estimators can track the parameter variation but the new FCT estimator performs much better.
\subsection{DT FCT-DREM with alertness preservation}
\lab{subsec42}
In this subsection we present the simulation results for the DT estimators of Propositions \ref{pro2} and \ref{pro4}.  The estimator gains were set to $\gamma = 2$, $c=1$, $d=1$, $T_{\tt D}=1$ and $T=0.5$.

The same simulation scenario of the CT schemes given above is reproduced here and the transient behaviors are shown in Figs. \ref{fig:fig1_DT} and \ref{fig:fig2_DT}. Essentially the same remarks made for the CT schemes of the previous subsection are applicable in the DT case.

\begin{figure}[tb]
	\centering
	\includegraphics[width=1\linewidth]{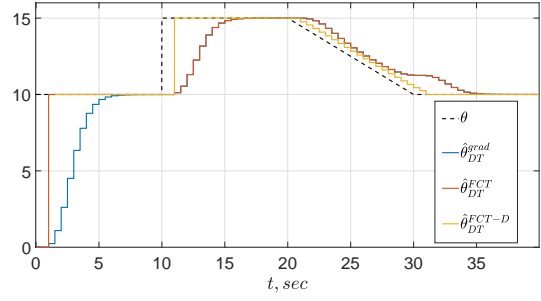}
	\caption{Transients of the DT parameter estimates with $\Delta(t)=\sin({\pi \over 10}t)$.}
	\label{fig:fig1_DT}
\end{figure}

\begin{figure}[tb]
	\centering
	\includegraphics[width=1\linewidth]{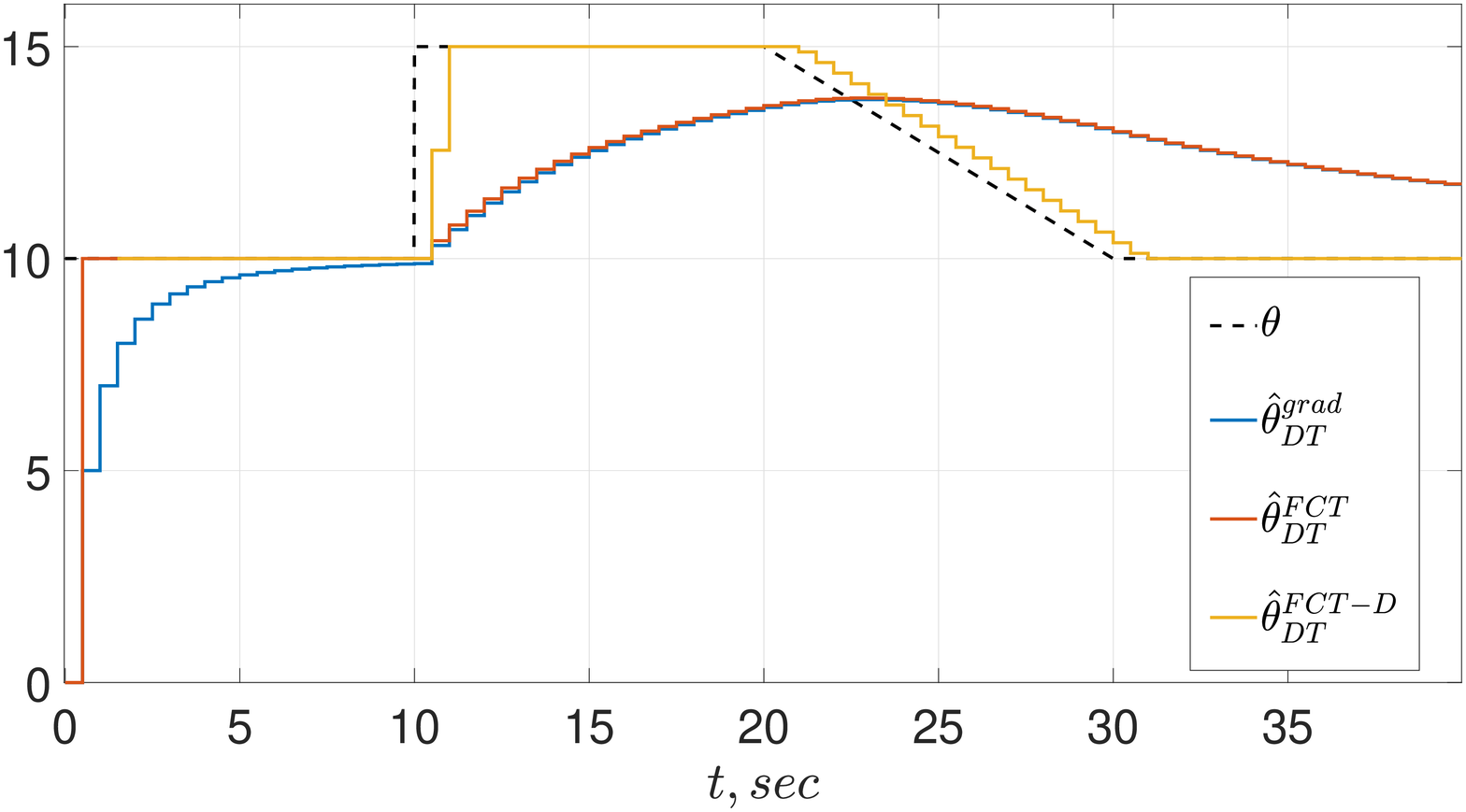}
	\caption{Transients of the DT parameter estimates with $\Delta(t)=\dfrac{1}{\sqrt{t+1}}$.}
	\label{fig:fig2_DT}
\end{figure}
\subsection{Comparison of the CT FCT-DREM with two schemes of  \cite{WANetal_tac20}}
\lab{subsec43}
Now, we compare the FCT-D algorithm with two of the schemes proposed in \cite{WANetal_tac20}. Namely,

\begin{enumerate}
	\item Algorithm 1
	\begin{align}
		\label{ef_9}
		\dot{\hat{\theta}}(t)=\gamma \Delta(t) \lceil \caly(t)-\Delta(t)\hat{\theta}(t) \rfloor ^\alpha
	\end{align}
	where $\gamma>0$ and $\alpha \in \left[ 0, 1\right) $.
	\item Algorithm 3
	\begin{align}
		\label{ef_13}
		\dot{\hat{\theta}}(t)=\gamma \sign(\Delta(t)) \lceil \caly(t)-\Delta(t)\hat{\theta}(t) \rfloor ^{\frac{\left|\Delta(t) \right| }{\varsigma \Delta_{max}}}
	\end{align}
	where $\gamma > 0$, $\varsigma > 1$, $\Delta_{max}=\max_{t\in \left[ -T^0,T^0+T \right] \left| \Delta(t) \right| }$ for given $T^0 \in \rea_+$ and $\hat{\theta}(t_0)=0$.
\end{enumerate}

It should be underscored in  \cite{WANetal_tac20} a third estimator was also proposed but this scheme is not well-defined when $\Delta(t)=0$ and could not be simulated.

For the simulation of the CT FCT-D we set $\gamma = 2$, $\mu=0.98$, and $T_{\tt D}=0.2$, while for Algorithm 1 \eqref{ef_9} we set $\gamma=5$ and $\alpha=0.75$.
and for Algorithm 3 \eqref{ef_13} we set  $\gamma = 5$, $\varsigma = 2$.

The same simulation scenario of the CT schemes given in Subsection \ref{subsec41} is reproduced here and the transient behaviors are shown in Figs. \ref{fig:ef_1} and \ref{fig:ef_2}. As shown in the figures Algorithm 3 indeed achieves FCT and preserves its alertness. However, in spite of the theoretical analysis reported in \cite{WANetal_tac20} the estimator of Algorithm 1 tracks only asymptotically.

To test the sensitivity of the algorithms to the presence of noise, we repeated the simulations adding a signal $0.1\sin(10t)$ to the measurement $\caly(t)$. The results of the simulations are shown in Figs. \ref{fig:ef_1_noise} and \ref{fig:ef_2_noise}. Interestingly, the behavior of Algorithm 1 does no seem to be affected by the noise---but, as before, the tracking is only asymptotic. For the PE case of Fig \ref{fig:ef_1_noise} we see only a minor performance degradation due to the noise in all schemes that is further degraded for the non-PE case of  Fig \ref{fig:ef_2_noise}. 

\begin{figure}[tb]
	\centering
	\includegraphics[width=1\linewidth]{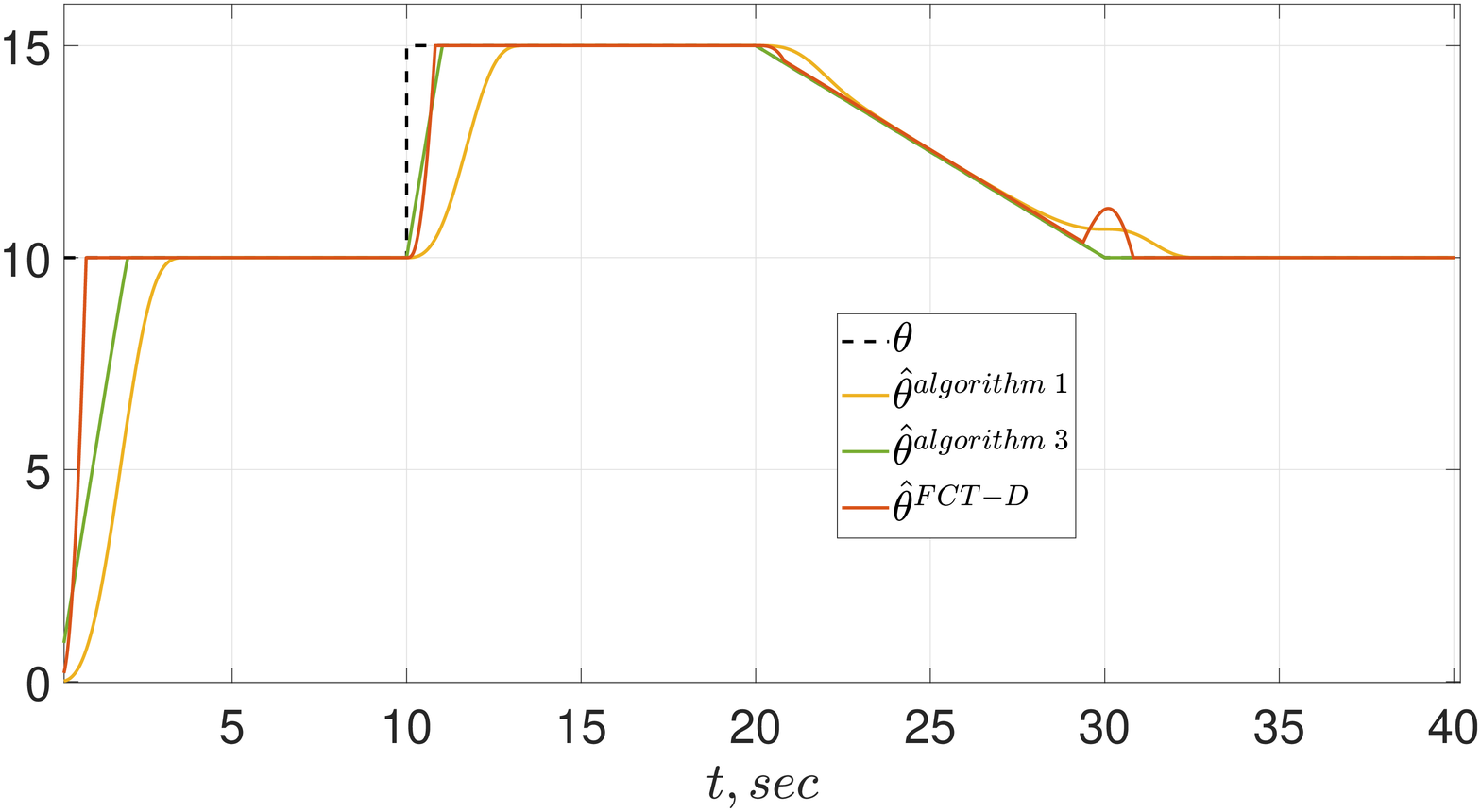}
	\caption{Transients of the parameter estimates with $\Delta(t)=\sin({\pi \over 10} t)$ for the CT FCT-D estimator and Algorithms 1 and 3 of  \cite{WANetal_tac20}. }
	\label{fig:ef_1}
\end{figure}

\begin{figure}[tb]
	\centering
	\includegraphics[width=1\linewidth]{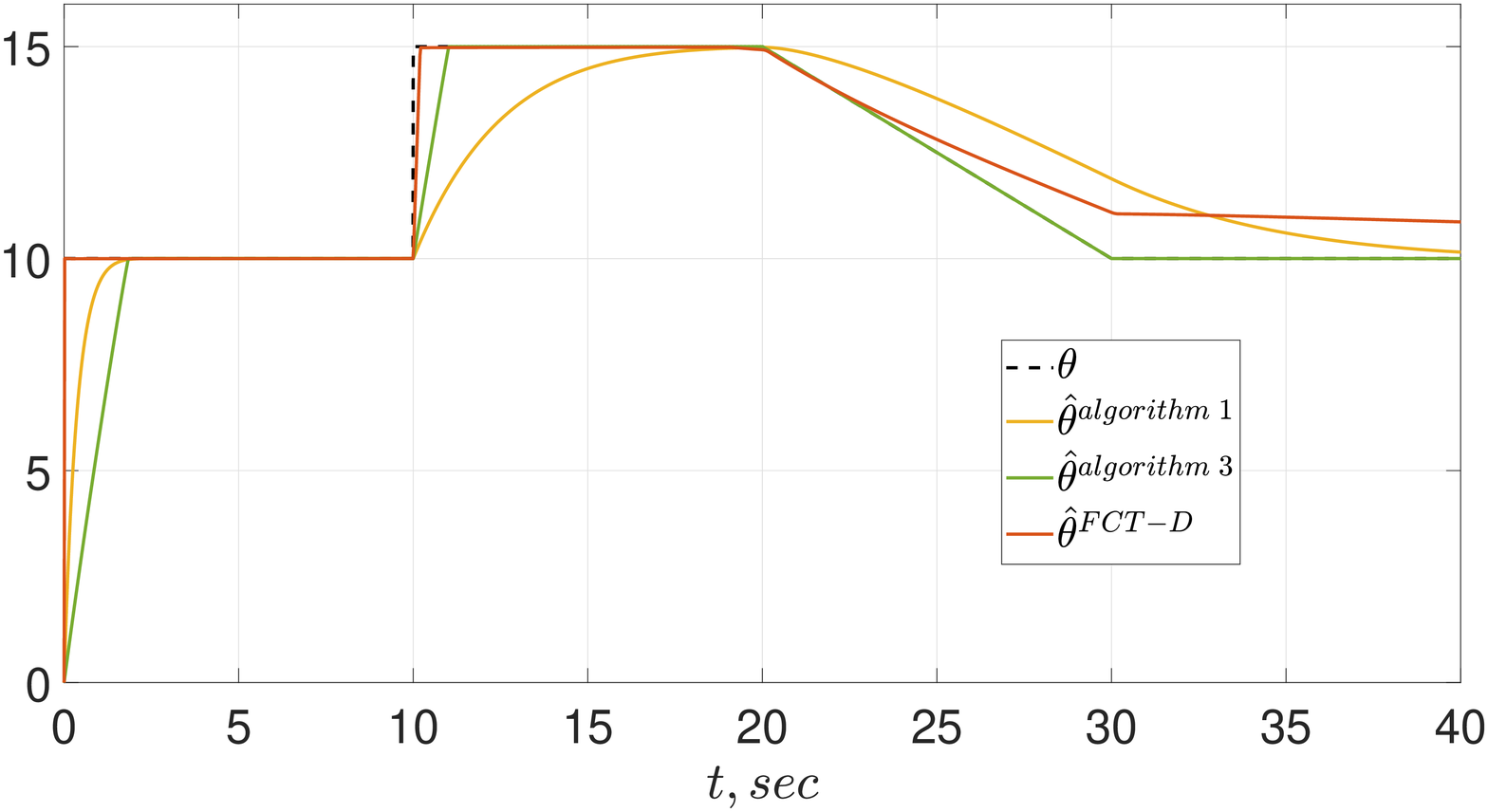}
	\caption{Transients of the parameter estimates with $\Delta(t)=\dfrac{1}{\sqrt{t+1}}$ for the CT FCT-D estimators and Algorithms 1 and 3 of \cite{WANetal_tac20}.}
	\label{fig:ef_2}
\end{figure}

\begin{figure}[tb]
	\centering
	\includegraphics[width=1\linewidth]{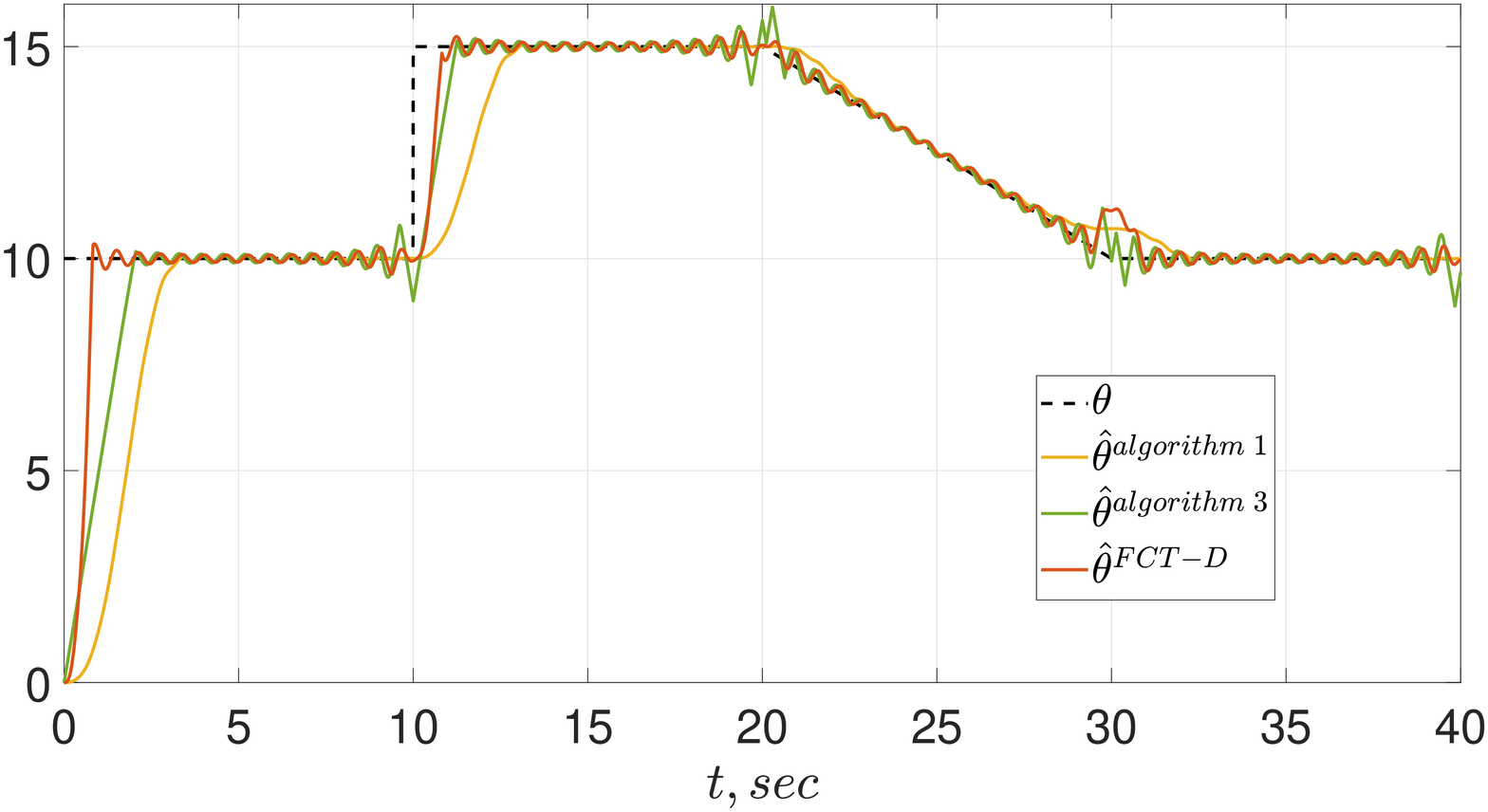}
	\caption{Transients of the parameter estimates with $\Delta(t)=\sin({\pi \over 10}t)$ for the CT FCT-D estimator and Algorithms 1 and 3 of \cite{WANetal_tac20} with noise.}
	\label{fig:ef_1_noise}
\end{figure}

\begin{figure}[tb]
	\centering
	\includegraphics[width=1\linewidth]{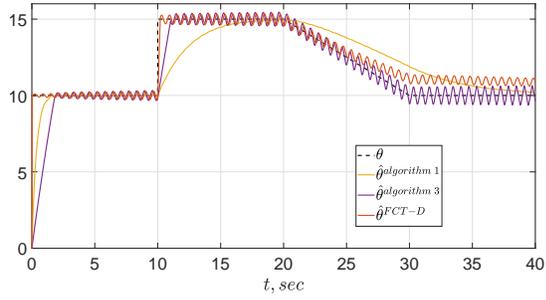}
	\caption{Transients of the parameter estimates with $\Delta(t)=\dfrac{1}{\sqrt{t+1}}$ for the CT FCT-D observer and Algorithms 1 and 3 of  \cite{WANetal_tac20}.}
	\label{fig:ef_2_noise}
\end{figure}

%
\section{CONCLUSIONS}
\lab{sec5}
%
We have presented two new FCT-DREM parameter estimators with enhanced performance---in particular, with respect to their ability to track parameter variations in {\em finite-time}. CT and DT versions of the new estimators are given. The performance improvement of the proposed schemes was illustrated with representative simulations.    

\appendix

\begin{table}[h]
	\centering
	\caption{List of Acronyms}
	\label{tab:2}
	\renewcommand\arraystretch{1.6}
	\begin{tabular}{l|r}
		\hline\hline
		AP & Alertness preservation \\
		CT &  Continuous-time  \\
		DREM  &  Dynamic regressor extension and mixing \\
		DT  & Discrete-time \\
		FCT & Finite-convergence time \\
		IE & Interval excitation \\
		LRE & Linear regressor equation \\
		\hline\hline
	\end{tabular}
\end{table}

\end{document}